\documentclass[12pt]{amsart}
\usepackage{hyperref}
\usepackage{amsfonts}
\usepackage{amsthm}
\usepackage{mathrsfs}
\usepackage{amssymb}
\usepackage{textcomp}
\usepackage{eucal}
\newtheorem{Pa}{Paper}[section]
\newtheorem{theorem}[Pa]{{\bf Theorem}}
\newtheorem{lem}[Pa]{{\bf Lemma}}
\newtheorem{corollary}[Pa]{{\bf Corollary}}

\newtheorem{prop}[Pa]{{\bf Proposition}}
\newtheorem{definition}[Pa]{{\bf Definition}}
\newtheorem{Ex}[Pa]{{\bf Example}}

\def\R{\mathbb R}

\def\(s){\mathscr S(\R^2)}

\usepackage{appendix}
\usepackage{amsmath}
\usepackage{color}
\usepackage{graphicx}
\usepackage{subfig}
\newcommand{\w}{\omega}

\title[Stochastic integration and an extension of $S$-transform]
{Stochastic integration for
a wide class of Gaussian stationary increment processes
using an extension of the $S$-transform}

\author{Daniel Alpay}
\address{(DA) Department of Mathematics
\newline
Ben Gurion University of the Negev \newline P.O.B. 653,
\newline
Be'er Sheva 84105, \newline ISRAEL}
\email{dany@math.bgu.ac.il}
\author{Alon Kipnis}
\address{(AK) Department of Mathematics
\newline
Ben Gurion University of the Negev \newline P.O.B. 653,
\newline
Be'er Sheva 84105, \newline ISRAEL}
\email{kipnisal@bgu.ac.il}
\keywords{stochastic integral, white noise space, fractional Brownian motion}
\subjclass{Primary: 60H40,60H05. Secondary:}
\thanks{D. Alpay thanks the
Earl Katz family for endowing the chair which supported his
research}

\begin{document}
\maketitle
\begin{abstract}
Given a Gaussian stationary increment processes with spectral density, we show that a
Wick-It\^{o} integral with respect to this process can be naturally obtained
using Hida's white noise space theory. We use the Bochner-Minlos theorem to associate a probability
space to the process, and define the counterpart of the $S$-transform in this space. We then use
this transform to define the stochastic integral and prove an associated It\^{o} formula.
\end{abstract}
\section{Introduction}
\setcounter{equation}{0}
Stochastic integration with respect to general Gaussian processes can be found in \cite[Ch.7]{MR99f:60082}. 
An extension of It\^{o} formula for such integrals has been derived  
by Nualart and Taqqu \cite{MR2220074,MR2456333} and recently in \cite{aal3} in the setting of white noise analysis. \\

In this paper we extend the $S$-transform approach to develop stochastic calculus and 
derive additional results for the family of centered Gaussian processes with covariance function of the form
\[
K_m(t,s)=\int_{\mathbb R}\frac{e^{i\xi t}-1}{\xi}\frac{e^{-i\xi s}-1}{\xi}m(\xi)d\xi
\]
where $m$ is a positive measurable even function subject to
\[
\int_{\mathbb R}\frac{m(u)}{\xi^2+1}d\xi<\infty.
\]

Note that $K_m(t,s)$ can also be written as
\[
K_m(t,s)=r(t)+r(s)-r(t-s),
\]
where
\[
r(t)=\int_{\mathbb R}\frac{1-\cos(t\xi)}{\xi^2}m(\xi)d\xi.
\]
This family includes in particular the fractional Brownian motion,
which corresponds (up to a multiplicative constant)
to $m(\xi)=|\xi|^{1-2H}$, where $H\in(0,1)$. We note that complex-valued functions of the form
\[
K(t,s)=r(t)+\overline{r(s)}-r(t-s)-r(0),
\]
where $r$ is a continuous function, have been studied in particular by von Neumann,
Schoenberg and Krein. Such a function  is positive definite if and only if $r$ can be written in
the form
\[
r(t)=r_0+i\gamma t+\int_{\mathbb R}\left\{e^{i \xi t}-1-\frac{i\xi t}{\xi^2+1}\right\}\frac{d\sigma(\xi)}{\xi^2},
\]
where $\sigma$ is an increasing right continuous function
subject to $\int_{\mathbb R}\frac{d\sigma(\xi)}{\xi^2+1}<\infty$. See \cite{MR0004644},
\cite{MR0012176}, and see \cite{aal2} for more information on these kernels.\\

As in \cite{aal2}, our starting point is the (in general unbounded) operator $T_m$ on the Lebesgue space of complex-valued
functions $\mathbf L_2(\mathbb R)$ defined by
\begin{equation} \label{def:Tm}
\widehat{T_mf}(\xi)=\sqrt{m}(\xi)\widehat{f}(\xi),
\end{equation}

with domain
\[
\mathcal D(T_m)=\left\{ f\in\mathbf L_2(\mathbb R)\, ; \int_{\mathbb R}m(\xi)|\widehat{f}(\xi)|^2d\xi<\infty\right\},
\]
where $\widehat{f}(\xi)=\frac{1}{\sqrt{2\pi}}
\int_{\mathbb R}e^{-i\xi t}f(t)dt$ denotes the Fourier transform.
Clearly, the Schwartz space $\mathscr S$ of smooth rapidly decreasing
functions belong to the domain of $T_m$. The indicator functions
\[
\mathbf 1_t=\begin{cases}\mathbf 1_{[0,t]},\quad t\ge 0,\\
                        \mathbf 1_{[t,0]}\quad t\le 0,\end{cases}
\]
also belong to $\mathcal D(T_m)$. In \cite{aal2}, and with some restrictions on $m$,
a centered Gaussian process $B_m$ with covariance function $K_m(t,s)=\left( T_m\mathbf 1_t,
T_m\mathbf 1_s\right)$ was constructed in Hida's white noise space.
In the present paper  we chose a different path. We build from $T_m$ the characteristic functional
\begin{equation}
\label{eq:tm}
C_m(s)=e^{-\frac{\|T_ms\|_{\mathbf L_2(\mathbb R)}^2}{2}}.
\end{equation}
It has been proved in \cite{MR2793121} that $C_m$ is continuous from $\mathscr S$ into $\mathbf L_2(\mathbb R)$.
Restricting $C_m$ to real-valued functions and using the Bochner-Minlos theorem, we obtain an analog of the white noise space in which the process $B_m$ is built in a natural way.
Stochastic calculus with respect to this process is then developed using an $S$-transform approach. \\

We note that when $m(\xi)=|\xi|^{1-2H}$, and $H\in(\frac{1}{2}, 1)$,
the operator $T_m$ reduces, up to a multiplicative constant,
to the operator $M_H$ defined in \cite{eh} and in
\cite{bosw}. The set $L_\phi^2$ presented in
\cite[equation (3)]{MR1741154}, is closely related to the domain of $T_m$, and the functional $C_m$ was used with the
Bochner-Minlos theorem in \cite[(3.5), p. 49]{MR2387368}.
In view of this, our work generalized the stochastic calculus for fractional
Brownian motion presented in these works to the aforementioned family of Gaussian processes.\\

Note moreover that the function $\phi$ from the last references
defines the kernel associated to the operator $T_m^*T_m$ via Schwartz' kernel theorem,
with $m=M_H$. In the general case, the kernel associated to the operator $T_m^*T_m$
is not a function. This last remark is the source for some of the difficulties arises
in extending Wick-It\^{o} integration for fractional Brownian motion such as the distinction between
the cases $H< \frac{1}{2}$, $H>\frac{1}{2}$ and $H=\frac{1}{2}$. An advantage
of our approach is that such issues do not arise.\\

There are two main ideas in this paper. The first is the construction of a probability Gaussian space in which a stationary increment process with spectral density $m$ is naturally defined. This result, being a concrete example of Kolmogorov's extension theorem on the existence of a Gaussian process with a given spectral density, is interesting in its own right.
The second main result deals with developing stochastic integration with respect to the fundamental process in this space. We take an approach based on the analog of the $S$-transform in our setting, and show that this stochastic integral coincides with the one already defined in \cite{aal3} but in the framework of Hida's white noise space.\\

Possible extensions and applications of the ideas presented here are 
in the field of stochastic control by considering systems driven by Gaussian noise of an arbitrary spectrum. The case of fractional noise was investigated for example in \cite{hu2005stochastic}, \cite{KlepLeb2002} and \cite{duncan2009control}. We argue that extensions of some of these works to systems driven by noises of a more general spectrum is straightforward in view of our present work. \\

The paper consists of five sections besides the introduction. In Section \ref{WAnoise} we
construct an analog of Hida's white noise space using the characteristic function $C_m$. In
Section \ref{sec:3} the associated fundamental process $B_m$ is being defined and studied.
The analog of the $S$-transform is defined and studied in Section \ref{sec4}. In Section \ref{sec:stoc_int}
we define a Wick-It\^{o} type stochastic integral with respect to $B_m$, and prove
an associated It\^o formula. In the last section we relate the present stochastic integral
with previous extensions of the It\^{o} integral for non semi-martingales.

\section{The $m$ Noise Space}
\label{WAnoise}
\setcounter{equation}{0}
We set $\mathscr S_{\mathbb R}$ to be the space of real-valued Schwartz functions,
and $\Omega=\mathscr S_{\mathbb R}^\prime$. We denote by $\mathcal B$ the
associated Borel sigma algebra. Throughout this paper, we denote by $\langle \cdot, \cdot \rangle$ the duality between $\mathscr S_{\mathbb R}^\prime$ and $\mathscr S_{\mathbb R}$, and by $\left(\cdot,\cdot\right)$ the inner product in $\mathbf L_2(\mathbb R)$. In case there is no danger of confusion, the $\mathbf L_2(\mathbb R)$ norm will be denoted as $\| \cdot \|$.

\begin{theorem} \label{th:probabilty_measure}
There exists a unique probability measure $\mu_m$ on
$(\Omega,\mathcal B)$ such that
\[
e^{-\frac{\|T_m s\|^2}{2}}=\int_{\Omega}e^{i\langle \omega,s\rangle}d\mu_m(\omega),
\quad s\in\mathscr S_{\mathbb R},
\]
\end{theorem}

{\bf Proof:} The function $C_m(s)$ is positive definite on
$\mathscr S_{\mathbb R}$ since
\[
C_m(s_1-s_2) = \exp\left\{-\frac{1}{2}
\|T_ms_1\|^2\right\}\times
\exp\left\{\left( T_ms_1,T_ms_2\right)
\right\} \times
\exp\left\{-\frac{1}{2} \|T_ms_2\|^2\right\}.
\]
Moreover the operator $T_m$ is continuous from
$\mathscr S$ (and hence from $\mathscr S_{\mathbb R}$)
into $\mathbf L_2(\mathbb R)$. This was proved
in  \cite{MR2793121}, and we repeat the argument for completeness. As in
\cite{MR2793121} we set $K=\int_{\mathbb R}\frac{m(u)}{1+u^2}du$ and
$s^\sharp(u)=\overline{s(-u)}$.  We have for $s\in\mathscr S$:
\[
\begin{split}
\|T_ms\|^2&=\int_{\mathbb R}|\widehat{s}(u)|^2m(u)du\\
&=\int_{\mathbb R}|(1+u^2)\widehat{s}(u)|^2\frac{m(u)}{1+u^2}du\\
&\le K\left(\int_{\mathbb R}|s\star s^\sharp|(\xi) d\xi+
\int_{\mathbb
R}|s^\prime\star(s^\sharp)^\prime|(\xi) d\xi\right)\\
&\le K\left(\left(\int_{\mathbb R}|s(\xi)|d\xi\right)^2+
\left(\int_{\mathbb R}|s^\prime(\xi)|d\xi\right)^2\right),
\end{split}
\]
where we have denoted convolution by $\star$.
Therefore $C_m$ is a continuous map from $\mathscr S_{\mathbb R}$
into $\mathbb R$, and the existence of $\mu_m$ follows
from the Bochner-Minlos theorem.\qed\mbox{}\\

The triplet $\left(\Omega,\mathcal B,\mu_m\right)$ will be used as our probability space.

\begin{prop}
Let $s\in\mathscr S_{\mathbb R}$. Then:
\begin{equation}
\label{eq:iso}
\mathbb E[\langle \omega, s\rangle^2]=\|T_ms\|^2_{\mathbf{L_2}(\mathbb R)}.
\end{equation}
\end{prop}

{\bf Proof:} We have

\begin{equation} \label{eq:1}
e^{-\frac{1}{2} \|T_ms\|^2}
{=} \int_{\Omega}
{e^{i\langle\w,s\rangle}d\mu_m(\w)}.
\end{equation}
Expanding both sides of (\ref{eq:1}) in power series we obtain

\begin{equation}\label{zeroMean}
\mathbb E\left[\langle\omega,s\rangle \right]= \int_{\Omega}
{\langle\w,s\rangle d\mu_m(\w)}=0.
\end{equation}
and
\begin{equation}
\begin{gathered}
\mathbb E\left[\langle\omega,s\rangle^{2}\right] {=} \int_{\mathscr S'}
\langle\w,s\rangle^2 d\mu_m(\w) = \langle T_ms,T_ms\rangle_{\mathbf{L_2}(\mathbb R)}.
\end{gathered}
\label{eq:isometry}
\end{equation}

We now want to extend the isometry \eqref{eq:iso} when $s$ is replaced by $f$
in the domain of $T_m$. This involves two separate steps: First, an
approximation procedure, and next complexification. The following two
propositions deal with the approximation.

\begin{prop} \label{prop:denseInRange}
The space $T_m(\mathscr S)$ is dense in the range of $T_m$.
\end{prop}

{\bf Proof:} Let $f\in{\mathcal D}(T_m)$ be such that $T_mf$ is orthogonal to $T_m(\mathscr S)$,
and let $h_n, n=1,2,\ldots $ denote the Hermite functions. Since these are
eigenvectors of the Fourier transform we have:
\[
\begin{split}
0&=
\left( T_mf,T_mh_n\right)\\
&=\left( \sqrt{m}\widehat{f},\sqrt{m}\widehat{h_n}\right)\\
&=c_n\left(\sqrt{m}\widehat{f},\sqrt{m}h_n\right)\\
&=c_n\left( m\widehat{f},h_n\right),
\end{split}
\]
where $c_n\in\mathbb R$ depends only on $n$.
So $m\widehat{f}\equiv 0$ since the $h_n$ form a basis
of $\mathbf L_2(\mathbb R)$.
Thus $\sqrt{m}\widehat{f}\equiv 0$ and so $T_mf=0$.
\mbox{}\qed\mbox{}\\

\begin{theorem}
The isometry \eqref{eq:iso} extends to $T_mf$ where $f$ is real-valued and in
the domain of $T_m$.
\end{theorem}

{\bf Proof:} We first note that, for  $f$ in the domain of $T_m$ we have
\begin{equation}
\label{sym}
\overline{T_mf}=T_m\overline{f}.
\end{equation}
Indeed, since $m$ is even and real  we have
\[
\widehat{T_m\overline{f}}=\sqrt{m}(\widehat{f})^\sharp=
(\sqrt{m}\widehat{f})^\sharp=\left(\widehat{T_mf}\right)^\sharp=\widehat{\overline{
T_mf}}.
\]

Let now $f$ be real-valued and in the domain of $T_m$,
and let $(s_n)_{n\in\mathbb N}$
be a sequence of elements in $\mathscr S$ such that
\begin{equation}
\label{lim}
\lim_{n\rightarrow\infty}\|T_ms_n-T_mf\|=0.
\end{equation}
In view of \eqref{sym}, and since
$f$ is real-valued we have
\begin{equation}
\label{lim1}
\lim_{n\rightarrow\infty}\|T_m\overline{s_n}-T_mf\|=0.
\end{equation}
Together with \eqref{lim} this last equation leads to
\begin{equation}
\label{lim11}
\lim_{n\rightarrow\infty}\|T_m({\rm Re}~s_n)-T_mf\|=0.
\end{equation}
In particular $(T_m({\rm Re}~s_n))_{n\in\mathbb N}$ is a
Cauchy sequence in $\mathbf{L_2}(\mathbb R)$.
By \eqref{eq:iso},
$(\langle \w,{\rm Re}~s_n\rangle)_{n\in\mathbb N}$ is a
Cauchy sequence in $\mathcal W_m$. We denote by
$\langle \w, f\rangle$ its limit. It is easily checked that
the limit does not depend on the
given sequence for which \eqref{lim} holds.
\mbox{}\qed\mbox{}\\

We denote by $\mathcal D_{\mathbb R}(T_m)$ the elements in the domain of
$T_m$ which are real-valued.\\

Let $f,g\in\mathcal D_{\mathbb R}(T_m)$. The polarization identity applied to
\begin{equation}
\label{eq:iso111}
\mathbb E[\langle \omega, f\rangle^2]=\|T_mf\|^2_{\mathbf{L_2}(\mathbb R)},
\quad f\in{\mathcal D}_{\mathbb R}(T_m).
\end{equation}
leads to
\[
\mathbb E\left[\langle \omega,f\rangle\langle \omega, g\rangle\right]=
{\rm Re}~\langle T_mf, T_mg\rangle_{\mathbf {L_2}(\mathbb R)}.
\]
In view of  \eqref{sym}, $T_mf$ and $T_mg$ are real and so we have:
\begin{prop}
Let $f,g\in\mathcal D_{\mathbb R}(T_m)$. It holds that
\begin{equation}
\label{Aisometrey}
\mathbb E\left[\langle \omega,f\rangle\langle \omega, g\rangle\right]=
\langle T_mf, T_mg\rangle_{\mathbf {L_2}(\mathbb R)}.
\end{equation}
\end{prop}

\begin{prop}
$\left\{
\langle \w,f\rangle
,\, f\in{\mathcal D}_{\mathbb R}(T_m)\right\}$ is a Gaussian process in
the sense that, for every linear
in the sense that for any $f_1,...,f_n \in
{\mathcal D}_{\mathbb R}(T_m) $ and $a_1,...,a_n \in \mathbb R$,
the random variable $\sum_{i=1}^{n}{a_i \langle \w,f_i \rangle}$ has a normal
distribution.
\end{prop}

\begin{proof}
By (\ref{eq:1}), for $\lambda \in \mathbb R$ we have,
\begin{equation}
\begin{split}
E[e^{i\lambda \sum_{i=1}^{n} {a_i \langle \w,f_i \rangle}}]&=\int_{\Omega}
{e^{i \lambda \sum_{i=1}^{n}{a_i\langle \w,f_i \rangle }}
d\mu_{m}(\omega)}\\
&=
\int_{\Omega} {e^{i \langle\omega, \lambda
\sum_{i=1}^{n}{a_i f_i} \rangle}} d\mu_{m}(\omega)\\
&=
e^{-\frac{1}{2} \lambda^2 \|\sum_{i=1}^{n}{a_iT_m f_i}
\|^2}.
\label{Gaussian}
\end{split}
\end{equation}
\end{proof}

In particular, we have that for any $\xi_1,...,\xi_n \in
\mathcal D_{\mathbb R}\left(T_m\right)$ such that $T_m\xi_1,...,T_m\xi_n$
are orthonormal in
$\mathbf L_2\left(\mathbb R\right)$ and for any $\phi \in \bf{L_2}(\mathbb R^n)$
\begin{equation}\label{Gaussiandistribution}
\mathbb E\left[\phi\left(\langle\omega,\xi_1\rangle,...\langle\omega,\xi_1\rangle
\right) \right]=\frac{1}{(2\pi)^{\frac{n}{2}}}\int_{\mathbb
R^n}{\phi(x_1,...,x_n)\prod^{n}_{i=1}{e^{-\frac{1}{2}{x_i}^2}}dx_1\cdot...\cdot
dx_n }.
\end{equation}

\begin{definition}
We set $\mathcal G$ to be the $\sigma$-field generated by $\left\{
\langle\omega,f\rangle, f \in \mathcal D_{\mathbb R}\left(T_m \right) \right\}$,
and denote $\mathcal W_m\triangleq
\mathbf{L_2}\left(\Omega,\mathcal G,\mu_m\right)$.
\end{definition}
 Note that $\mathcal G$ may be
significantly smaller than $\mathcal B$, the Borel $\sigma$-field of $\Omega$. For
example, if $m\equiv0$, then $T_m$ is the zero operator and
$\mathcal G =\left\{\emptyset,\Omega,0,\Omega
\backslash \{0\} \right\}$. \newline
We will see in the following section that the time
derivative, in the sense of distributions, of the fundamental stochastic
process $B_m$ in the space $\mathcal W_m$ has spectral density $m(\xi)$. It is therefore justify to refer $\mathcal W_m$ as the $m$-{\sl noise space}. \\
In the case $m\left(\xi\right)\equiv1$, $T_m$ is
the identity over $\mathbf{L_2}\left(\mathbb R\right)$ and $\mu_m$ is
the white noise measure used for example in \cite[(1.4), p. 3]{MR1244577}.
Moreover, by Theorem 1.9 p. 7 there, $\mathcal G$ equals the Borel sigma algebra and
so the $1$-noise space coincides with Hida's white noise space.

\section{The process $B_m$}
\setcounter{equation}{0}
\label{sec:3}
We now define a process $B_m\,\,:\,\,\mathbb R\longrightarrow\mathcal W_m$ via
\[
B_m(t)\triangleq B_m(t,\omega)\triangleq \langle\omega,\mathbf 1_t\rangle.
\]
where $\mathbf 1_A$ is the indicator function of the set $A$ and
$\mathbf 1_t \triangleq \mathbf 1_{[0,t]}$.
This process plays the role of the
Brownian motion for the Ito formula in the space $\mathcal W_m$. Note that this is
the same definition as the Brownian motion in \cite{MR1408433}, the difference being the
probability measure assigned to $\left(\Omega,\mathcal
B\right)$.\newline

\begin{theorem} \label{Amotion}
$B_m$ has the following properties:
\begin{enumerate}
\item
$B_m$ is a centered Gaussian random process.
\item
for $t,s \in \mathbb R$, the covariance of $B_m(t)$ and
$B_m(s)$ is
\[
K_m(t,s)=\int_{\mathbb R}\frac{e^{i\xi t}-1}{\xi}\frac{e^{-i\xi s}-1}{\xi}m(\xi)d\xi
=\left( T_m\mathbf 1_t,T_m \mathbf 1_s\right).
\]
\item
The process $B_m$ has a continuous version under the condition
\[
\int_{\mathbb R}\frac{m(\xi)}{1+|\xi|}d\xi<\infty
\]

\end{enumerate}
\end{theorem}
\begin{proof}[Proof:]
(1) follows from (\ref{Gaussian}) and (\ref{zeroMean}).\\
To prove (2), we see that by (\ref{Aisometrey}) we have
\[
\begin{split}
\mathbb E\left[B_m(t)B_m(s)\right]&=\mathbb E\left[\langle\omega,\mathbf 1_t\rangle
\langle\omega,\mathbf 1_s\rangle\right]\\
&={\rm Re}~ \left( T_m \mathbf 1_t,T_m \mathbf 1_s\right) \\
&=\left( T_m \mathbf 1_t,T_m \mathbf 1_s\right),
\end{split}
\]
since this last expression is real.\\
The covariance function $K_m$ is a particular case of the covariance function
$K_\sigma$ presented in \cite{MR2793121} and (3) is a consequence of Theorem 10.1 there.
\end{proof}

We bring here two interesting examples for specific choices of the spectral density $m$ and the resulting process $B_m$.
\begin{Ex}[filtered white noise]
\label{ex:colored_noise}
Consider the spectral density
\[
m_1\left(\xi\right)=\mathbf 1_{[-\Delta,\Delta]},\quad \Delta\geq0.
\]
The corresponding process $B_{m_1}$ has the covariance function
\[
 K_{m_1}(t,s)=\frac{1}{\sqrt{2\pi}}\int_{-\Delta}^{\Delta}\frac{1-\cos(t\xi)-\cos(s\xi)+\cos(\xi(t-s))}{\xi^2}d\xi,\quad t,s \in \mathbb R.
\]
The time derivative of this process also belongs to $\mathcal W_m$, and is a stationary Gaussian process with
covariance
\begin{equation} \label{eq:fractional_colored_noise}
\frac{\partial^2}{\partial t \partial s}K_{m_1}\left(t,s\right)=\frac{1}{\sqrt{2\pi}}\int_{-\Delta}^{\Delta}e^{i(t-s)\xi}d\xi=\frac{2\sin\left(\Delta(t-s) \right)}{t-s}.
\end{equation}
This process can be obtained in physical models by passing a white noise through a low-pass filter with cut-off frequency $\Delta$. We can see from
the covariance function \eqref{eq:fractional_colored_noise} that each time sample $t_0\in \mathbb R$ is positively correlated with time samples in the interval $\left[t_0,t_0+\frac{\pi}{2\Delta}\right)$, negatively correlated with time samples in the interval $\left(t_0+\frac{\pi}{2\Delta},t_0+\frac{3\pi}{2\Delta}\right)$ and so on with decreasing magnitude of correlation. This behavior may describes well price fluctuation of some financial asset.
\end{Ex}

\begin{Ex}[band limited fractional Brownian motion]
We can combine the spectral density $m_1$ from the previous example with the spectral density $|\xi|^{1-2H}$ of the fractional noise with Hurst parameter $H \in \left(0.5,1\right)$ to obtain a process with covariance function
\[
 K_{m_2}(t,s)=\frac{1}{\sqrt{2\pi}}\int_{-\Delta}^{\Delta}\frac{1-\cos(t\xi )-\cos(s\xi)+\cos\left( (t-s)\xi \right)}{\xi^2}|\xi|^{1-2H} ,\quad t,s \in \mathbb R.
\]
This process shares both properties of long range dependency of the fractional Brownian motion with Hurst parameter $H$ and the ripples of the filtered noise for its time derivative. As the bandwidth $\Delta$ approaches infinity, the covariance function $K_{m_2}$ uniformly converges (up to a multiplicative constant) to the covariance of the fractional Brownian motion.
\end{Ex}

Our next goal is to develop stochastic calculus based on the
process $B_m$ in the space $\mathcal W_m$. The definition of the Wiener integral with respect to $B_m$ for
$f \in \mathcal D\left(T_m \right)$ is straightforward and given by
\begin{equation}
 \int_0^{\tau}{f(t)dB_m(t)\triangleq \langle\omega,f \rangle}.
\end{equation}

Note that since
\[
 \int_{\mathbb R} m(\xi)|\widehat{f}(\xi)|^2d\xi =
\leq \sup_{\xi \in \mathbb R}(1+\xi^2)|\widehat{f}(\xi)|^2 \int_{\mathbb R}\frac{m(\xi)}{1+\xi^2}d\xi < \infty,
 \]
a sufficient condition for a function $f\in \mathbf{L_2}\left(\mathbb R\right)$ to be in the domain of $T_m$ is
\[
  \sup_{\xi \in \mathbb R}(1+\xi^2)|\widehat{f}(\xi)|^2\leq \infty.
\]
This is satisfies in particular if $f$ is differentiable with derivative in $\mathbf{L_2}\left(\mathbb R\right)$.

Recall that in the white noise space one defines the It\^{o}-Hitsuda stochastic integral of $X_t$ on the interval $[a,b]$ as
\[
\int_a^b X_t\diamond \dot{B}_m dt
\]
where $\dot{B}_m$ denotes the time derivative of the Brownian motion and $\diamond$ denotes the Wick product.
The chaos decomposition of the white noise space is used in order to define the Wick product
and appropriate spaces of stochastic distributions where $\dot{B}_m$ lives.\\
Chaos decomposition for $\mathcal W_m$ can be obtained by a similar procedure to the one carried in \cite[section 3]{MR1976868,eh} for the fractional Brownian motion. A space of stochastic distributions that contains $\dot{B}_m$ and is closed under the Wick product can similarly be defined.
However, some inconvenience arises when one tries to obtain a chaos decomposition for $\mathcal W_m$
since any basis to yield it  explicitly depends on the spectral density $m$ through $T_m$. Moreover,
the time derivative of the process $B_m$ may already exists as an element of $W_m$
as Example \ref{ex:colored_noise} teaches us. For those reasons we find that an approach based on the
analog of the $S$ transform in our setting is more general and natural since it uses only the
expectation and the Lebesgue integral on the real line.

\section{The $S_m$ transform}
\setcounter{equation}{0}
\label{sec4}
We now define the analog of the $S$ transform in the
space $\mathcal W_m$ and study its properties. Some of the results
can be carried out immediately from properties of the $S$ transform of
the white noise space from \cite{MR1244577} for example,
but others require more attention.
\newline

For $s \in \mathscr S_{\mathbb R}$ we define the analog of the Wick
exponential in the space $\mathcal W_m$:
\[
:e^{\langle\omega,s\rangle}: \triangleq
e^{\langle\omega,s\rangle-\frac{1}{2} \|T_ms \|^2}
\]

\begin{definition}
The $S_m$ {transform} of $\Phi \in \mathcal W_m$
is defined by
\[
(\mathcal S_m\Phi)(s)\triangleq
\int_{\Omega}{:e^{\langle\omega,s\rangle}:\Phi(\omega) d\mu_m(\omega)}=\mathbb E\left[
:e^{\langle \omega,s\rangle}: \Phi(\omega) \right], \quad s\in \mathscr
S_{\mathbb R}.
\]
\end{definition}

The following result is similar to \cite[Theorem 2.2]{MR2046814}.

\begin{theorem}\label{S-surjection}
Let $\Phi,\Psi \in \mathcal W_m$.
If $\left(\mathcal S_m\Phi\right)(s)=\left(\mathcal S_m\Psi\right)(s)$ for
all $s\in\mathscr S_{\mathbb R}$, then $\Phi=\Psi$.
\end{theorem}
\begin{proof}[Proof:]
By linearity of the $S_m$ transform, it is enough to
prove
\[
\forall s\in \mathscr S,\quad \left(\mathcal S_m \Phi \right)(s)=0
\Rightarrow \Phi=0.
\]
 Let $\left\{\xi_n\right\}_{n \in \mathbb N}\subset \mathscr S_{\mathbb R}$ be a
countable dense set in $\bf{L_2}(\mathbb R)$ and denote by $\mathcal G_n$ the
$\sigma$-field
generated by $\left\{\langle \omega,\xi_1 \rangle,...,\langle \omega,\xi_n\rangle
\right\}$.
We may choose $\left\{\xi_n\right\}_{n \in \mathbb N}$ such that $\left\{T_m\xi_n\right\}_{n \in \mathbb N}$ are orthonormal. For every
$n \in \mathbb N$, $E\left[\Phi|\mathcal G_n
\right]=\phi_n\left(\langle\omega,\xi_1\rangle,...,
\langle \omega,\xi_n\right\rangle)$ for some measurable function
$\phi_n\, :\, \mathbb R^n \longrightarrow \mathbb R$ such that
\[
\mathbb E\Phi=\idotsint\limits_{\mathbb R^n}{\phi_n(x)
e^{-\frac{1}{2} x'x}dx}< \infty,
\]
where $x^\prime$ denotes the transpose of $x$; see for instance
\cite[Proposition 2.7, p. 7]{MR0264757}. Thus, for $t=(t_1,...,t_n)
\in \mathbb R^n$, using \eqref{Gaussiandistribution} we obtain
\begin{flalign*}
0 &=\int_{\Omega}{:e^{\langle \omega,\sum_{k=1}^{n}{t_k
\xi_k}\rangle}:\Phi(\omega)d\mu_m}=\int_{\Omega}{:e^{\langle \omega,\sum_{k=1}^{n}{t_k
\xi_k}\rangle}:\mathbb E\left[\Phi|\mathcal G_n
\right]d\mu_m}(\w) \\
&= e^{-\frac{1}{2}\sum_{k=1}^{n}{t_k \|T_m\xi_k}\|^2}
\int_{\Omega}{e^{\sum_{k=1}^{n}t_k{\langle\omega,\xi_k}\rangle}
\phi_n\left(\langle\omega,\xi_1\rangle,...,\langle\omega,\xi_n\rangle\right)d\mu_m}(\w)\\
&= e^{-\frac{1}{2}\sum_{k=1}^{n}{t_k \|T_m\xi_k}\|^2}
\frac{1}{(2\pi)^{\frac{n}{2}}}\idotsint\limits_{\mathbb R^n} {e^{\sum_{k=1}^{n}t_k x_k}
\phi_n\left(x_1,...,x_n\right)e^{-\frac{1}{2}\sum_{k=1}^{n}x_k^2}dx_1...dx_n}\\
&=\idotsint\limits_{\mathbb R^n}{
\phi_n\left(x\right)e^{-\frac{1}{2}(x-t)'(x-t)}dx}.
\end{flalign*}
By properties of the Fourier transform, we get that $\phi_n=0$ for all $n\in \mathbb N$. Since $ \bigcup_{{n\in\mathbb N}} {\mathcal G_n}=\mathcal G$ we have $\Phi=0$.
\end{proof}

\begin{definition}
A stochastic exponential is a random variable of the form
\[
e^{\langle\omega,f\rangle}, \quad f \in \mathcal D_{\mathbb R}\left(T_m\right).
\]
We denote by $\mathscr E$ the family of linear combinations of stochastic
exponentials.
\end{definition}

Since $:e^{\langle\omega,f\rangle}:=e^{-\frac{1}{2}\| T_mf \|^2}
e^{\langle\omega,f\rangle}$, the following claim is a direct consequence of Theorem \ref{S-surjection}.

\begin{prop}
\label{cor:1}
$\mathscr E$ is dense in ${\mathcal W_m}$.
\end{prop}

\begin{definition}
A stochastic polynomial is a random variable of the form
\[
p\left({\langle\omega,f_1\rangle},...,{\langle\omega,f_2\rangle}\right), \quad
f_1,...,f_n \in \mathcal D_{\mathbb R}\left(T_m\right).
\]
for some polynomial $p$ in $n$ variables.
We denote the set of stochastic polynomials by $\mathscr P$.
\end{definition}

\begin{corollary} \label{cor:2}
The set of stochastic polynomials is dense in ${\mathcal W_m}$.
\end{corollary}
\begin{proof}
We first note that the stochastic polynomials indeed belong to
${\mathcal W_m}$ because the random variables $\langle \w, f\rangle$ are
Gaussian and hence have moments of any order.\\

Let $\Phi \in {\mathcal W_m}$ such that $\mathbb E\left[\Phi p \right]=0$ for
each $p\in \mathscr P$. Then any $f \in \mathcal D_{\mathbb R}(T_m)$,
\begin{equation}
E\left[e^{\langle\omega,f\rangle}\Phi(\omega) \right]=E\left[\sum_{n=0}^\infty
{\frac{\langle\omega,f\rangle^n}{n!} }\Phi(\omega) \right]=
\sum_{n=0}^\infty {\frac{\mathbb E\left[\langle\omega,f\rangle^n \Phi(\omega) \right]}
{n!}}=0,
\end{equation}
where interchanging of summation is justified by Fubini's theorem since
\[
\begin{split}
\sum_{n=0}^\infty {\mathbb E\left[\mid \frac{\langle\omega,f\rangle^n}{n!}
\Phi(\omega) \mid \right]}& \leq
\sum_{n=0}^\infty {\frac{1}{(n!)}\sqrt{ \mathbb E\left[ \langle\omega,f
\rangle^{2n}\right] \mathbb E\left[ {\Phi(\omega)}^2 \right]}}\\
&\leq \sum_{n=0}^\infty \sqrt{\frac{\left(2n-1\right)!!}{(n!)^2}}
{\parallel T_mf \parallel}_{\mathbf L_2\left(\mathbb R\right)}^{n}
\sqrt{ \mathbb E\left[ {\Phi(\omega)}^2 \right]} \\
&\leq \sum_{n=0}^\infty {\frac{2^n}{n!} {\parallel T_mf \parallel
}_{\mathbf L_2\left(\mathbb R\right)}^{n}
\sqrt{ \mathbb E\left[ {\Phi(\omega)}^2 \right]}}\\
&= e^{2{\parallel T_mf \parallel}_{\mathbf L_2\left(\mathbb R\right)}^2 } \cdot \sqrt{ \mathbb E\left[
{\Phi(\omega)}^2 \right]} <\infty.
\end{split}
\]
(We have used the Cauchy-Schwarz inequality and the moments of a Gaussian distribution).\\
We have showed that $\mathbb E\left[e^{\langle\omega,f\rangle}\Phi(\omega) \right]=0$
for any $f \in \mathcal D_{\mathbb R}(T_m)$ so by Theorem \ref{S-surjection}
we obtain $\Phi=0$ in ${\mathcal W_m}$.
\end{proof}

\begin{lem}\label{Sidentity}
Let $f,g\in {\mathcal D}_{\mathbb R}(T_m)$. Then
\[
E[:e^{\langle\omega,f\rangle}:~
:e^{\langle\omega,g\rangle}:]=e^{\langle T_mf,T_mg\rangle_{\mathbf L_2\left(\mathbb R\right)}}.
\]
\end{lem}
\begin{proof}[Proof:]
\begin{equation}
E[:e^{\langle\omega,f\rangle}:]=e^{-\frac{1}{2}\|T_mf\|^2}E[e^{\langle\omega,f\rangle}]=1,
\end{equation}
since $E[e^{\langle\omega,f\rangle}]$ is the moment generating
function of the Gaussian random variable
$\langle \omega,f\rangle$ with variance
$\|T_mf\|^2$ valued at $1$.\newline
 Thus we get
\[
\mathbb E[:e^{\langle\omega,f\rangle}:~
:e^{\langle\omega,g\rangle}:]=e^{\langle T_mf,T_mg\rangle_{\mathbf L_2\left(\mathbb R\right)}}
\mathbb E[:e^{\langle\omega,f+g\rangle}:]=e^{\langle T_mf,T_mg\rangle_{\mathbf L_2\left(\mathbb R\right)}}.
\]
\end{proof}

The following formula is useful in calculating the $S_m$ transform of a multiplication of two random variables,
and can be easily proved using Lemma \ref{Sidentity}.
\begin{equation} \label{eq:S-multiplication}
\mathcal S_m \left(:e^{\langle \omega,f\rangle}:~:e^{\langle \omega,g\rangle}: \right)=e^{\left( T_ms,T_mf \right)} e^{\left( T_ms,T_mf\right)}e^{\left( T_ms,T_mg\right)}
,\quad f,g\in \mathcal D_{\mathbb R}(T_m).
\end{equation}

\begin{prop} \label{S-convergence}
Let $\left\{\Phi_n\right\}$ be a sequence in ${\mathcal W_m}$ that converges in ${\mathcal W_m}$ to $\Phi$. Then for
any $s\in \mathscr S_{\mathbb R}$ the
sequence of real numbers $\left\{ \mathcal S_m \left(\Phi_n \right)(s)
\right\}$
converges to $\mathcal S_m \left(\Phi \right)(s)$.
\end{prop}
\begin{proof}
For any $s \in \mathscr S_{\mathbb R}$,
\[
|\mathcal S_m \Phi_n (s)-\mathcal S_m \Phi (s)|=
|\mathbb E\left[:e^{\langle\omega,s\rangle}:
( \Phi_n-\Phi) \right]|\leq
\sqrt{\mathbb E\left[\left(:e^{\langle\omega,s\rangle}:\right)^2 \right]}\cdot
\sqrt{ \mathbb E\left[\left(\Phi_n-\Phi\right)^2 \right]}.
\]
By direct calculation $\mathbb E\left[\left(:e^{\langle\omega,s\rangle}:\right)^2 \right]=e^{\| T_ms \|^2}$ and since $\mathbb E\left[\left(\Phi_n-\Phi\right)^2
\right] \longrightarrow 0$, the claim follows.
\end{proof}

In the statement of Theorem \ref{translation}
recall that $T_m$ is a bounded operator from $\mathscr S_{\mathbb R}$ into
$\mathbf L_2(\mathbb R)$ and so its adjoint is a bounded operator
from $\mathbf L_2(\mathbb R)$ into $\mathscr S_{\mathbb R}^\prime$.

\begin{theorem} \label{translation}
For $\Phi \in \mathcal W_m$ and $s\in \mathscr S_{\mathbb R}$,
\[
\mathcal S_m\Phi (s)=\int_\Omega {\Phi(\omega+T_m^*T_ms)d\mu_m(\omega)}.
\]
\end{theorem}
\begin{proof}

Assume first that $\Phi\left(\omega\right)=:e^{\langle \omega, s_1\rangle}:$
where $s_1 \in \mathscr S_{\mathbb R}$. We have by Lemma \ref{Sidentity} that
\[
\begin{split}
\int_\Omega{ \Phi \left(\omega+T_m^*T_ms\right) d\mu_m\left(\omega\right)}
&=\int_\Omega {:e^{\langle
\omega,s_1 \rangle}:e^{\langle T_m^*T_ms, s_1\rangle}
d\mu_m\left(\omega\right)}\\
&=e^{\langle T_m^*T_ms,s_1\rangle}
\int_\Omega {:e^{\langle \omega,s_1 \rangle}: d\mu_m\left(\omega\right)}\\ &=
 e^{\left( T_ms,T_ms_1\right) }\cdot 1\\&=\mathcal S_m\Phi (s).
\end{split}
\]
The result may be extended by linearity to any $\Phi \in \mathcal E$
which is a dense subset of ${\mathcal W_m}$ by Propositions \ref{cor:1}
and \ref{S-convergence}.
\end{proof}

We can find the $S_m$ transform of powers of $\langle \omega, f\rangle $
for $f\in\mathcal D_{\mathbb R}(T_m)$ by the formula for Hermite polynomials.
\begin{corollary}
For $f\in\mathcal D_{\mathbb R}(T_m)$ and $s\in\mathcal S_{\mathbb R}$,
we have that
\begin{equation}\label{HermiteRelation}
\left( T_ms,T_mf \right)^n=n!
\sum_{m=0}^{\lfloor n/2 \rfloor}\frac{\left(-\frac{1}{2} \right)^m
\left(\mathcal S_m\langle\omega,f \rangle^{n-2m} \right)(s)~
\|T_mf\|^{2m}}{m!(n-2m)!},
\end{equation}

in particular
\begin{equation} \label{WeinerIntegral}
(\mathcal S_m \langle\omega,f\rangle)(s)=\left( T_mf,T_ms \right)
\end{equation}
and
\begin{equation}\label{Ssquare}
(\mathcal S_m \langle\omega,f\rangle^2)(s)=\left( T_mf,T_ms\right)^2+\|T_ms\|^2
\end{equation}
\end{corollary}\label{Spowers}
\begin{proof}
From Lemma \ref{Sidentity} we get that
\[
(\mathcal S_m:e^{\langle\omega,f\rangle}:)(s)=e^{\left( T_ms,T_mf \right)},
\]
then,
\begin{equation}
e^{-\frac{1}{2}\|T_mf\|^2} \mathcal S_m \left(
\sum_{k=0}^{\infty}{\frac{\langle\omega,f\rangle^k}{k!}}
\right)(s)=\sum_{k=0}^{\infty}{\frac{\left( T_ms,T_mf\right)^k}{k!}}
\end{equation}
By the linearity of the $S_m$ transform and Fubini's
theorem, and replacing $f$ by $tf$ with $t\in\mathbb R$ we compare powers of $t$ at both sides to get
\eqref{HermiteRelation}.
\end{proof}

This last corollary can be also formulated in terms of the Hermite polynomials. Recall that the $n_{th}$ Hermite polynomial with parameter $t \in \mathbb R$ is defied by
\begin{equation}
h_{n}^{\left[t\right]}\left(x\right)\triangleq  n! \sum_{m=0}^{\lfloor n/2 \rfloor}\frac{\left(-\frac{1}{2} \right)^m  x^{n-2m} \cdot t^{2m}}{m!(n-2m)!}
\end{equation}
(see for instance \cite[p. 33]{MR1387829}).
For $f \in {\mathcal D}(T_m)$ we define
\begin{equation}
\tilde{h_n}\left(\langle\omega,f\rangle\right)
\triangleq h_n^{\left[\|T_mf\|\right]}\left(\langle\omega,f\rangle\right)=
  n! \sum_{m=0}^{\lfloor n/2 \rfloor}\frac{\left(-\frac{1}{2} \right)^m
\langle\omega,f \rangle^{n-2m} \cdot \|T_mf\|^{2m}}{m!(n-2m)!},
\label{hntilde}
\end{equation}
and we also set $\tilde{h}_0=1$.\\

So by (\ref{HermiteRelation}) we have that
\begin{equation} \label{Hermite Powers}
\left(\mathcal S_m \tilde{h}_n\left(\langle\omega,f\rangle \right) \right)(s)=\left(T_ms,T_mf \right)^n.
\end{equation}

Using Equation \ref{HermiteRelation} and Lemma \ref{Sidentity},
one can easily verify the following result:

\begin{prop}
Let $f\in {\mathcal D}_{\mathbb R}(T_m)$. It holds that:
\begin{equation}
:e^{\langle\omega,f\rangle}:=\sum_{k=0}^{\infty}{\frac{\tilde{h}_k
\left(\langle\omega,f\rangle \right)}{k!}}
\end{equation}
\end{prop}

It is possible to define a Wick product in $\mathcal W_m$ using
the $S_m$ transform.
\begin{definition}\label{Wick product}
Let $\Phi,\Psi \in \mathcal W_m$.
The Wick product of $\Phi$ and $\Psi$ is the element $\Phi \diamond \Psi
 \in \mathcal W_m$ that satisfies
\[
\mathcal S_T(\Phi \diamond \Psi)(s)=(\mathcal S_T\Phi)(s)(\mathcal
S_T \Psi)(s),\quad \forall s\in\mathscr S_{\mathbb R},
\]
if it exists.
\end{definition}
As this definition suggests, in general the Wick product is
not stable in $\mathcal W_m$.\\

From \eqref{Hermite Powers}, the Wick product of Hermite polynomials satisfies
\[
\tilde{h}_n\left(\langle\omega,f\rangle\right) \diamond \tilde{h}_k\left(\langle\omega,f\rangle\right)
=\tilde{h}_{n+k}\left(\langle\omega,f \rangle\right), \quad n,k \in \mathbb N,\quad f\in{\mathcal D_{\mathbb R}}(T_m).
\]

\section{The stochastic integral}
\label{sec:stoc_int}
\setcounter{equation}{0}
We now use the $S_m$-transform to define a Wick-It\^{o} stochastic integral and
prove an It\^{o} formula for this integral.
In the next section we also show that for particular choice of $m$,
our definition of the stochastic integral coincide with previously defined
Wick-It\^{o} stochastic integrals for fractional Brownian motion; see
\cite{MR1741154,MR2387368}.
We set
\[
\mathcal B_s(t)= \mathcal S_m \left(B_m(t)\right)(s).
\]

By \eqref{WeinerIntegral} we see that
\begin{equation}
\label{S-B_m}
\mathcal B_s(t)
=\left( T_ms,T_m\mathbf 1_t\right)=\int_{\mathbb R}m(\xi)\widehat{s}(\xi)\frac{e^{i\xi t}-1}{\xi}d\xi.
\end{equation}
This function is absolutely continuous with respect to Lebesgue measure and its derivative is
\begin{equation}
\label{eq:mprime}
(\mathcal B_s(t))^\prime=\int_{\mathbb R}m(\xi)\widehat{s}(\xi)e^{i\xi t}d\xi .
\end{equation}
We note that when $T_m$ is a bounded operator from
$\mathbf L_2(\mathbb R)$ into itself we have
by a result of Lebesgue (see \cite[p. 410]{sch_III}), $(\mathcal B_s(t))^\prime=(T_m^*T_ms)(t)$ (a.e.).\\

\begin{definition}\label{def:StochasticIntegral}
Let $M\in\mathbb R$ be a Borel set
and $X:M  \longrightarrow \mathcal W_m$ a
stochastic process. The process $X$ will be called integrable over $M$
if for any $s\in \mathscr S_{\mathbb R}$,
$\left(\mathcal S_m X_t\right)(s) \mathcal B_s(t)^\prime$
is integrable on $M$, and if
there is a $\Phi \in \mathcal W_m$ such that
\[
\mathcal S_m \Phi (s)=\int_M{ \left(\mathcal S_m X_t \right)(s)\mathcal B_s(t)^\prime dt}.
\]
for any $s\in \mathscr S_{\mathbb R}$. If $X$ is integrable, $\Phi$ is uniquely determined by
Theorem \ref{S-surjection} and we denote it by $\int_M{X_t dB_m\left(t\right)}$.
\end{definition}

If $T=Id_{L_2\left(\mathbb R\right)}$, this definition coincides with
the \textit{Hitsuda-Skorohod} integral \cite[Chapter 8]{MR1244577}. See also
Section \ref{sec:stoc_int}.\\

Note that since
\[
|\mathcal B_s(t)'|\leq\int_{\mathbb R}m(\xi)|\widehat{s}(\xi)|d\xi\leq \sup_\xi|(1+\xi^2)\widehat{s}(\xi)|  \int_{\mathbb R}\frac{m(\xi)}{1+\xi^2}d\xi <\infty,
\]
for any $s\in \mathscr S$ there exist a constant $K_s$ such that
\begin{flalign*}
|\int_M \mathcal S_m X_t(s) \mathcal B_s(t)^\prime dt| & \leq K_s \int_M |\mathbb E\left[X_t:e^{\langle\omega,s\rangle}:\right] |dt \\
& \leq K_s \mathbb E\left[:e^{\langle\omega,s\rangle}:\right] \int_M E|X_t| dt.
\end{flalign*}
Thus a sufficient condition for the integrability of $\mathcal S_m X_t(s) \mathcal B_s(t)^\prime$ is
$\int_M E|X_t|dt<\infty$. \newline
We note that general conditions for the integrability of a stochastic process on $\mathcal W_m$ cannot
be easily obtained. Inconvenient conditions for integrability is the price we pay for not relying on
stochastic distributions.

\begin{theorem} \label{th:non-random}
Any non-random $f \in \mathcal D_{\mathbb R}(T_m)$ is integrable and we
have,
\begin{equation}
\int_0^\tau{f(t)dB_m(t)}=\langle\omega,\mathbf 1_{\left[0,\tau \right]}f
\rangle.
\end{equation}
\end{theorem}
\begin{proof}
In virtue of (\ref{WeinerIntegral}) and the definition of the
stochastic integral, we need to show that
\[
\int_0^\tau f(t) \mathcal B_s(t)^\prime dt=\left( T_ms,T_m\mathbf 1_{
\left[0,\tau \right]} f\right).
\]
Using formula \eqref{eq:mprime} and Fubini's theorem, we have:
\[
\begin{split}
\int_0^\tau f(t) \mathcal B_s(t)^\prime dt&=\int_0^\tau f(t) \left(
\int_{\mathbb R}m(\xi)\widehat{s}(\xi)e^{i\xi t}d\xi\right)dt\\
&=\int_{\mathbb R} m(\xi)\widehat{s}(\xi)\left(\int_0^\tau f(t)e^{it\xi}dt
\right)d\xi\\
&=\int_{\mathbb R} m(\xi)\widehat{s}(\xi)\left(\widehat{
f\mathbf 1_{[0,\tau]}}(\xi)\right)d\xi\\
&=\left( T_ms,T_m\mathbf 1_{
\left[0,\tau \right]} f\right).
\end{split}
\]

\end{proof}

\begin{prop}
The stochastic integral has the following properties:
\begin{enumerate}
\item For $0\leq a<b \in \mathbb R$,
\[
B_m\left(b\right)-B_m\left(a\right)=\int_a^b{dB_m(t)}
\]
\item Let $X:M\longrightarrow \mathcal W_m$ an integrable process. Then
\[
\int_M{X_tdB_m(t)}=\int_{\mathbb R}{\mathbf 1_M  X_t dB_m(t)}.
\]
\item Let $X:M\longrightarrow \mathcal W_m$ an integrable process. Then
\[
\mathbb E\left[\int_M{X_t dB_m(t)} \right]=\mathcal S_m \left(\int_M{X_t dB_m(t)} \right)(0)=0
\].
\item The Wick product and the stochastic integral can be interchanged:
Let $X:\mathbb R \longrightarrow \mathcal W_m$ an integrable process and assume that for $Y\in \mathcal W_m$, $Y \diamond X_t$ is integrable. Then,
\[
Y \diamond \int_{\mathbb R}{X_t dB_T(t)}=\int_{\mathbb R}{Y \diamond
X_t dB_T(t)}
\]
\end{enumerate}
\end{prop}

\begin{proof}
The proof of the first three items is easy and we omit it. The last item is proved
in the following way:
\begin{flalign*}
\mathcal S_m(Y\diamond \int_{\mathbb R}{X_t dB_m(t)})(s)
 &=(S_mY)(s)\int_M{(\mathcal S_m X_t)(s) dB_T}\\
&= \int_M{(S_mY)(s)(\mathcal S_m X_t)(s) dB_m}\\
 &=\mathcal S_m(\int_{\mathbb R}{Y \diamond X_t dB_m(t)})(s).
\end{flalign*}
\end{proof}

\begin{Ex}For $\tau \geq 0$, we have by equation \eqref{Ssquare},
\[
\begin{split}
\int_0^{\tau}\left(T_ms,T_m\mathbf 1_{t}\right)
\frac{d}{dt}\left(T_ms,T_m\mathbf 1_{t}\right)dt & =
\frac{1}{2}\left(T_ms,T_m\mathbf 1_{\tau}\right)^2 \\
& = \frac{1}{2}\mathcal S_m \left( \langle \omega, \mathbf 1_{t}\rangle-\|T_m\mathbf 1_{t}\|^2 \right)(s).
\end{split}
\]
Then $B_m$ is integrable on the interval $[0,\tau]$, and we have
\[
\int_0^\tau B_m(t)dB_m(t)=\frac{1}{2}B_m(\tau)^2-\frac{1}{2}\|T_m\mathbf 1_{\tau}\|^2.
\]
This reduces to the well known result if $m\left(\xi\right)\equiv 1$ and
 $T_m$ is then the identity operator.
\end{Ex}

\begin{Ex}
Let $\widetilde{h_n}$ be defined by \eqref{hntilde}. A similar argument to the one in \eqref{eq:mprime} will show that any $f$ such that for any $f\mathbf 1_t \in \mathcal D(T_m)$ we have that the function $t\mapsto \left(T_ms,T_m f\mathbf 1_t\right)$ is differentiable with time derivative
\[
\frac{d}{dt}\left(T_ms,T_m f\mathbf 1_t\right)=f(t)\int_{\mathbb R} m(\xi)\widehat{s}(\xi)e^{-it\xi}d\xi=
f(t)\mathcal B_s(t)^\prime.
\]
By a similar argument to Theorem \ref{th:non-random} we have
\begin{flalign*}
\frac{1}{n+1} \mathcal S_m \left(\tilde{h}_{n+1}\left( \langle\omega ,\mathbf 1_\tau f \rangle\right)\left(t\right)\right)(s) & = \frac{1}{n+1} \left( T_ms,T_m \mathbf 1_\tau f\right) \\
& = \int_0^\tau  \left(T_ms,T_m f\mathbf 1_t\right)^n f(t) \mathcal B_s(t)^\prime dt \\
& = \mathcal S_m\left(\int_0^{\tau}f(t) \tilde{h}_n\left( \langle\omega ,\mathbf 1_t f\rangle\right)dB_m\left(t\right) \right)(s).
\end{flalign*}

Thus,
\begin{equation}
\label{eq_hn}
 \int_0^{\tau}f(t) \tilde{h}_n\left( \langle\omega ,\mathbf 1_t f\rangle  \right)dB_m
\left(t\right)=
\frac{1}{n+1}
\tilde{h}_{n+1}\left( \langle\omega ,\mathbf 1_\tau f\rangle  \right).
\end{equation}
\end{Ex}

It follows from \eqref{eq_hn} that for any polynomial $p$ and $f$ with $\mathbf 1_tf\in \mathcal D(T_m)$ the process
$p(\langle\omega ,\mathbf 1_t f\rangle)$ is integrable. This result can be easily extended to the process
\[
t\mapsto e^{\langle \omega, \mathbf 1_t f\rangle}, \quad \mathbf 1_tf\in \mathcal D(T_m),
\]
and we also obtain
\begin{corollary}
\[
\int_0^\tau f(t):e^{\langle \omega,\mathbf 1_t f\rangle }:dB_m(t)= :e^{\langle \omega,\mathbf 1_t f \rangle}:-1.
\]
\end{corollary}

\begin{Ex}
Let $f \in \mathcal D_{\mathbb R}(T_m)$. Using \eqref{eq:S-multiplication} we can obtain
\[
\begin{split}
\mathcal S_m \left(:e^{\langle \omega, f\rangle }:\int_0^\tau :e^{\langle \omega,\mathbf 1_t \rangle }:dB_m(t) \right)(s) & = \mathcal S_m \left(:e^{\langle \omega, f\rangle}:~:e^{\langle \omega, \mathbf 1_\tau \rangle}:-:e^{\langle \omega, f\rangle}: \right)(s)\\
& = e^{\left(T_ms,T_mf\right)}\left(e^{\left(T_ms,T_m\mathbf 1_\tau\right)}e^{\left(T_mf,T_m\mathbf 1_\tau\right)}-1\right).
\end{split}
\]
On the other hand,
\[
\begin{split}
\mathcal S_m \left( \int_0^\tau :e^{\langle \omega, f\rangle}:~:e^{\langle \omega,\mathbf 1_t \rangle}: dB_m(t) \right)(s) & =e^{\left(T_ms,T_mf \right)} \int_0^\tau e^{\left(T_ms,T_m\mathbf 1_t \right)} e^{\left(T_mf,T_m\mathbf 1_t \right)} \frac{d}{dt}\left(T_ms,T_m \mathbf 1_t\right)dt \\
&=e^{\left(T_ms,T_mf\right)}\left(e^{\left(T_ms,T_m\mathbf 1_\tau\right)}e^{\left(T_mf,T_m\mathbf 1_\tau\right)}-1\right)\\
& - \int_0^\tau e^{\left(T_ms,T_m \mathbf 1_t \right)}e^{\left(T_mf,T_m \mathbf 1_t \right)}\frac{d}{dt}\left(T_mf,T_m \mathbf 1_t\right)dt.
\end{split}
\]
So in general for an integrable stochastic process $X$ and a random variable $Y$,
\[
Y \int_0^\tau X_tdB_m(t)\neq \int_0^\tau Y X_tdB_m(t).
\]

\end{Ex}

\section{It\^{o}'s formula}
\label{Th:Ito_Formula}
\setcounter{equation}{0}
In this section we prove an It\^{o}'s formula. We begin by proving
an extension of the classical Girsanov theorem to our setting.
\begin{theorem}\label{Girsanov}
Let $f \in\mathcal D(T_m)$, and let $\mu$ be the measure defined by
$\mu(A)=E[:e^{\langle \omega,T_mf \rangle}:\mathbf 1_A]$. The process
\[
\tilde B_m(t) \triangleq B_m(t)-(T_mf,T_m\mathbf 1_t),
\]
is Gaussian and satisfies
\[
\mathbb E_{\mu}[\tilde B_m (t) \tilde B_m(s)]=
(T_m\mathbf 1_t,T_m \mathbf 1_s).
\]
\end{theorem}
\begin{proof}[Proof:]

We will first prove that for all $t \geq 0$, $\tilde B_m(t)$ is a
Gaussian random variable relative to the measure $\mu$ by considering its
moment generating function $\mathbb E_{\mu}\left[e^{\lambda \tilde B_m(t)}
\right]$, $\lambda \in \mathbb R$,

\begin{equation} \label{Girsanov_1}
\begin{split}
\mathbb E_{\mu}\left[e^{\lambda \tilde B_m(t)}
\right]&=\mathbb E\left[e^{\langle\omega,T_mf\rangle-\frac{1}{2}\|T_mf\|^2}
e^{\lambda \langle\omega,\mathbf 1_t\rangle
-\lambda\left(T_mf,\mathbf 1_t \right)}  \right]\\
&=
e^{-\lambda\left(T_mf,\mathbf 1_t\right)}~
e^{-\frac{1}{2}\|T_mf\|^2}
\mathbb E\left[e^{\langle\omega,T_mf+\lambda\mathbf 1_{t}\rangle} \right].
\end{split}
\end{equation}

Since $ \langle\omega,\phi+\lambda \mathbf 1_t \rangle$ is a zero mean
Gaussian random variable with variance
\[
\|T_m\left(f+\lambda \mathbf 1_t \right)\|^2=\|T_mf\|^2+\lambda^2\|T_m\mathbf 1_t\|^2+2\lambda\left(T_m\phi,T_m\mathbf 1_t\right),
\]
its moment generating function evaluated at $1$ is given by
\begin{equation}
\mathbb E\left[e^{\langle\omega,T_mf+\lambda\mathbf 1_t \rangle} \right]=e^{\frac{1}{2}
\|T_mf\|^2} e^{\frac{1}{2}\lambda^2\|T_m\mathbf 1_t\|^2}
e^{\lambda\left(T_mf,T_m\mathbf 1_t\right)},
\end{equation}
and we conclude from (\ref{Girsanov_1}) that
\begin{equation}
\mathbb E_{\mu}\left[e^{\lambda \tilde B_m(t)}
\right]=e^{\frac{1}{2}\lambda^2\|T_m\mathbf 1_t\|^2}.
\end{equation}

Thus for all $t\geq0$, $\tilde B_m(t)$ is a zero mean Gaussian random
variable on $\left(\Omega,\mathcal G,\mu_m \right)$. Similar arguments
will show that any linear combination of time samples is a Gaussian
variable, and thus $\widetilde B_m(t), t\ge 0$ is a Gaussian process.
Finally, by the polarization formula,
\begin{flalign*}
\mathbb E_{\mu}[\tilde B_m (t) \tilde B_m(s)]=
(T_m
\mathbf 1_{[0,t]},T_m \mathbf 1_{[0,s]}).
\end{flalign*}
\end{proof}

We now interpret integrals of the type $\int_0^\tau \Phi(t)dt$,
where for every $t\in[0,\tau]$, $\Phi(t)\in\mathcal W_m$, as Pettis integrals,
that is as\[
\mathbb E\left[ \left(\int_0^\tau \Phi(t)dt\right) \Psi\right]=
\int_0^\tau \mathbb E[\Phi(t)\Psi] dt,\quad\forall \Psi\in
\mathcal W_m,
\]
under the hypothesis that the function $t\mapsto
\mathbb E[\Phi(t)\Psi]$ belongs to
$\mathbf L_1([0,\tau],dt)$ for every $\Psi\in
\mathcal W_m$. See \cite[pp. 77-78]{Hille_Phillips}.
We note that if $X$ is moreover pathwise integrable and such that the pathwise
intregral belongs to $\mathcal W_m$, then
\[
\int_0^\tau \mathbb E[|X_t|]dt<\infty,
\]
and we can apply Fubini's theorem
to show that both integrals coincide. It is also clear from the definition
of the Pettis integral that it commutes with the $S_m$ transform.
\\

We define the conditions
\begin{eqnarray}
\label{eq1}
\mathbb E\left[ |F(t,X_t)|: e^{\langle \omega, s\rangle}:\right]&<&\infty\\
\label{eq2}
\mathbb E\left[ |\frac{\partial F}{\partial t}(t,X_t)|: e^{\langle \omega,
s\rangle}:\right]&<&\infty\\
\mathbb E\left[ |\frac{\partial F}{\partial x}(t,X_t)|:
e^{\langle \omega, s\rangle}:\right]&<&\infty
\label{eq3}
\end{eqnarray}
for $F\in C^{1,2}\left(\left[0,\infty
\right),\mathbb R\right)$.

We shall now develop an It\^{o} formula for a class of stochastic
processes of the form,
\begin{equation}\label{Weiner Integral}
X_t\left(\omega\right)=\int_0^\tau {f(t)dB_m(t)}=\langle
\omega,\mathbf 1_{[0,\tau]}f\rangle, \quad \tau\geq0,\quad
I_{[0,\tau]}f \in \mathcal D\left(T_m\right).
\end{equation}

\begin{theorem}
Let $F\in C^{1,2}\left(\left[0,\infty
\right),\mathbb R\right)$ satisfying \eqref{eq1}-\eqref{eq3}, and assume
that the function
$\|T_m\mathbf 1_t f|^2$
is absolutely continuous with respect to
the Lebesgue measure as a function of $t$. Then we have,

\begin{flalign} \label{Ito-Formula}
F(\tau,X_\tau)-F(0,0)=\int_0^\tau{\frac{\partial}{\partial
t}F(t,X_t)dt}+\int_0^\tau
{ f(t) \frac{\partial}{\partial x} F\left(t,X_t\right) dB_m \left(t\right)} \\
+\frac{1}{2}\int_0^\tau{\frac{d}{dt}\|T_m\mathbf
1_t f\|^2 \frac{\partial^2}{\partial x^2}
F(t,X_t)dt}  \nonumber
\end{flalign}
in $\mathcal W_m$.
\end{theorem}
This proof is based on \cite[Section 13.5]{MR1387829}.
\begin{proof}
Let $s\in\mathscr S_{\mathbb R}$ and $f\in{\mathcal D}~(T_m)$.
It follows from Theorem \ref{Girsanov} that for every $t\in[0,\tau]$,
$X_t(\w)=\langle\omega,\mathbf 1_t f \rangle$ is normally distributed under the measure
\[
\mu_s(A) \triangleq \mathbb E\left[
\mathbf 1_A \exp\left\{\langle\omega,s
\rangle-\frac{1}{2}\|T_ms\|^2 \right\} \right]=\mathbb E\left[\mathbf 1_A
:e^{\langle\omega,s \rangle} \right],
\]
with mean
\[
\left(T_ms,T_m\mathbf 1_{[0,t]}f\right)
\]
and variance
\[
\|T_m\mathbf 1_{[0,t]}f\|^2.
\]
Thus,
\begin{flalign} \label{S-Ito}
\left( \mathcal S_m
F(t,X_t)\right)(s)&=\mathbb E\left[:e^{\langle\omega,s\rangle}:F(t,X_t)\right]\\
                &=\int_{\mathbb R}{F\left(t,u+\left( T_m\mathbf 1_{[0,t]}f,{T_m}s
\right)\right)\rho\left(\|{T_m}\mathbf 1_{[0,t]}f
\|^2,u
                \right)du},\nonumber
\end{flalign}
where $\rho(w,u)=\frac{1}{\sqrt{2\pi w}}e^{-\frac{u^2}{2w}}$ and
satisfies,
\begin{equation} \label{rho-Ito}
\frac{\partial}{\partial
w}\rho=\frac{1}{2}\frac{\partial^2}{\partial u^2}\rho.
\end{equation}
Integrating by part we obtain:
\begin{equation} \label{parts-Ito}
\int_{\mathbb R}{F(t,u) \frac{\partial^2}{\partial u^2}\rho(w,u)
du}=\int_{\mathbb R}{\frac{\partial^2}{\partial u^2}F(t,u)\rho(w,u)}du.
\end{equation}
In view of \eqref{eq1}-\eqref{eq3} we may differentiate under the integral sign by
(\ref{S-Ito}),(\ref{rho-Ito}) and (\ref{parts-Ito}) and obtain for
$0\leq t \leq \tau$,
\begin{flalign*}
\frac{d}{dt}\mathcal S_m \left(F(t,X_t) \right)(s) & =\int_{\mathbb
R}{\frac{\partial}{\partial t} F\left(t,u+\left({T_m}\mathbb
\mathbb 1_{[0,t]}f,{T_m}s \right) \right) \rho \left(\|{T_m}\mathbb
\mathbb 1_{[0,t]}f\|^2,u \right)du}  \\
&\hspace{-1.3cm}+ \int_{\mathbb R}{\frac{\partial}{\partial x}
F\left(t,u+\left({T_m}\mathbf 1_{[0,t]}f,{T_m}s \right)
\right)\frac{d}{dt}\left({T_m}\mathbf 1_{[0,t]}f,{T_m}s \right) \rho
\left(\|{T_m}\mathbf 1_{[0,t]}f\|^2,u \right)du} \\
& \hspace{-1.3cm}
+ \int_{\mathbb R}{ F\left(t,u+\left({T_m}\mathbf 1_{[0,t]}f,{T_m}s
\right) \right) \frac{d}{dt}\|{T_m}\mathbf 1_{[0,t]}f\|^2
\frac{\partial}{\partial t} \rho
\left(\|{T_m}\mathbf 1_{[0,t]}f\|^2 \right)du} \\
& =  \mathcal S_m \left(\frac{\partial}{\partial t}
F\left(t,X_t\right) \right)(s)+\frac{d}{dt}\left({T_m}s,{T_m}\mathbb
\mathbf 1_{[0,t]}f \right)\mathcal S_m \left(\frac{\partial}{\partial x}
F\left(t,X_t\right) \right)(s)\\
 &\hspace{5mm} + \frac{1}{2}\frac{d}{dt}
\|{T_m}\mathbf 1_{[0,t]}f\|^2\cdot
\mathcal S_m \left(\frac{\partial^2}{\partial x^2}
F\left(t,X_t\right) \right)(s).
\end{flalign*}
Hence,
\begin{flalign}\label{eq:ItoLastStep}
\mathcal S_m \left(F\left(\tau,X_\tau\right)-F\left(0,0\right)
\right)(s) & =\int_{0}^{\tau}{\mathcal S_m
\left(\frac{\partial}{\partial t}F\left(t,X_t \right)\right)(s)dt} \\ \nonumber
& + \int_{0}^{\tau}{\frac{d}{dt}\left(T_ms,T_m\mathbf 1_{[0,t]}f
\right) \mathcal S_m \left(\frac{\partial}{\partial x}F\left(t,X_t
\right)\right)(s)dt} \\ \nonumber
& +\frac{1}{2} \int_{0}^{\tau}{\frac{d}{dt}\|{T_m}\mathbb
\mathbf 1_{[0,t]}f\|^2 \cdot \mathcal S_m
\left(\frac{\partial^2}{\partial x^2}F\left(t,X_t
\right)\right)(s)dt}.
\end{flalign}
Note that,
\[
\begin{split}
\mathcal S_m\left(\int_0^\tau {f(t) \frac{\partial}{\partial x}F\left(t,X_t
\right) dB_{m}(t)}\right)(s)&=\\
&\hspace{-2cm}=\int_{0}^{\tau}{\frac{d}{dt}\left(T_ms,T_m\mathbf 1_{[0,t]}f
\right)
 \mathcal S_m \left(\frac{\partial}{\partial x}F\left(t,X_t
\right)\right)(s)dt}.
\end{split}
\]
Thus we may use Fubini's theorem to interchange the $S_m$-transform and the pathwise integral, and
obtain that the $S_m$-transform of the right hand side of (\ref{Ito-Formula}) is exactly
the right hand side of (\ref{eq:ItoLastStep}) and the theorem is proved.\end{proof}

\section{Relation with other white-noise extensions of Wick-It\^{o} integral}
\setcounter{equation}{0}
Recall that the white noise space correspond to $m(\xi)\equiv 1$, so denoting it $\mathcal W_1$ is consistent with our notation, and $\mathcal S_1$ is the classical $S$-transform of the white noise space.
We define a map $\widetilde{T_m}: \mathcal W_m \longrightarrow \mathcal W_1$ by describing its action on the dense set of stochastic polynomials in $\mathcal W_m$:
\[
\widetilde{T_m} \langle \omega,f \rangle^n = \langle \omega, T_m f \rangle^n,\quad f\in \mathcal D\left(T_m\right).
\]
Note that since $\mathcal D\left(T_m\right)\subset {\mathbf L_2}\left(\mathbb R\right)$ this map is well defined.
It easy to see that $\widetilde{T_m}$ is an isometry of Hilbert spaces. By continuity we obtain that
\[
\widetilde{T_m}e^{\langle \omega,f\rangle}=e^{\langle \omega T_mf\rangle}, \quad f\in \mathcal D\left(T_m\right),
\]
hence
\[
\left(\mathcal S_1 \widetilde{T_m}{e^{\langle \omega,f\rangle}} \right)(T_m s)=e^{\left(T_ms,T_mf \right)}=\left(\mathcal S_m e^{\langle \omega,f\rangle}\right)(s).
\]
So this relations between the $S_1$ and the $S_m$ transform is extended such that for any $\Phi \in \mathcal W_m$,
\[
\left(\mathcal S_1 \widetilde{T_m}\Phi \right)(T_m s)=\left(\mathcal S_m \Phi\right)(s).
\]
Let $X:\left[0,\tau\right] \longrightarrow \mathcal W_m$
be a stochastic process. We have defined its It\^{o} integral as the unique element $\Phi\in \mathcal W_m$ (if exists) having $S_m$-transform
\[
\left(\mathcal S_m\Phi\right)(s)=\int_0^\tau \left(X_t \right)(s)\frac{d}{dt}\left(T_ms,T_m\mathbf 1_t\right)(s)dt.
\]
This suggests that if we define in the white noise the process $\tilde{B}_m$ as $\langle \omega, T_m \mathbf 1_t \rangle$ and stochastic integral with respect to $\tilde{B}_m$ as the unique element $\Phi \in \mathcal W_1$(if exists) having $S_1$-transform
\begin{equation} \label{def:whiteNoiseIntegral}
\left(\mathcal S_1\Phi\right)(s)=\int_0^\tau \left(X_t \right)(s)\frac{d}{dt}\left(s,T_m\mathbf 1_t)\right)(s)dt,
\end{equation}
both definitions coincides in the sense that
\begin{equation} \label{eq:integralRelation}
\widetilde{T_m}\int_0^\tau X_t dB_m(t)=\int_0^\tau \widetilde{T_m}X_t dB_m(t).
\end{equation}
Recall that the fractional brownian motion can be obtained in our setting by taking $m\left(\xi\right)=\frac{1}{2}|\xi|^{1-2H}$, $H\in(0,1)$, which results in $T_m=M_H$, where $M_H$ defined in \cite{eh}. In the white noise, the fractional Brownian motion can be defined by the continuous version of the process $\left\{\langle \omega,M_H \mathbf 1_t \rangle \right\}_{t\geq0}$. \\
An approach based on the definition described in \eqref{def:whiteNoiseIntegral} for the fractional Brownian motion was given in \cite{MR2046814}. Due to Theorem $3.4$ there, under appropriate conditions our definition of the It\^{o} integral in the case of $T_m=M_H$ coincide in the sense of \eqref{eq:integralRelation} with the Hitsuda-Skorohod integral. Stochastic integration in the white noise setting for the family of stochastic processes considered in this paper can be found in \cite{aal3}, and its equivalence to the integral described here
can be obtained by a similar argument to that of Theorem $3.4$ in \cite{MR2046814}.

\bibliographystyle{plain}


\begin{thebibliography}{10}

\bibitem{aal3}
D.~Alpay, H.~Attia, and D.~Levanony.
\newblock White noise based stochastic calculus associated with a class of
  {G}aussian processes.
\newblock Arxiv manuscript arXiv:1008.0186v1. (2010). To appear in {Opuscula
  Mathematica}.

\bibitem{aal2}
D.~Alpay, H.~Attia, and D.~Levanony.
\newblock On the characteristics of a class of {G}aussian processes within the
  white noise space setting.
\newblock {\em Stochastic processes and applications}, 120:1074--1104, 2010.

\bibitem{MR2793121}
D.~Alpay, P.~Jorgensen, and D.~Levanony.
\newblock A class of {G}aussian processes with fractional spectral measures.
\newblock {\em J. Funct. Anal.}, 261(2):507--541, 2011.

\bibitem{MR2046814}
Christian Bender.
\newblock An {$S$}-transform approach to integration with respect to a
  fractional {B}rownian motion.
\newblock {\em Bernoulli}, 9(6):955--983, 2003.

\bibitem{bosw}
F.~Biagini, B.~{\O}ksendal, A.~Sulem, and N.~Wallner.
\newblock An introduction to white-noise theory and {M}alliavin calculus for
  fractional {B}rownian motion, stochastic analysis with applications to
  mathematical finance.
\newblock {\em Proc. R. Soc. Lond. Ser. A Math. Phys. Eng. Sci.},
  460(2041):347--372, 2004.

\bibitem{MR2387368}
Francesca Biagini, Yaozhong Hu, Bernt {\O}ksendal, and Tusheng Zhang.
\newblock {\em Stochastic calculus for fractional {B}rownian motion and
  applications}.
\newblock Probability and its Applications (New York). Springer-Verlag London
  Ltd., London, 2008.

\bibitem{MR0264757}
R.~M. Blumenthal and R.~K. Getoor.
\newblock {\em Markov processes and potential theory}.
\newblock Pure and Applied Mathematics, Vol. 29. Academic Press, New York,
  1968.

\bibitem{MR1741154}
T.E. Duncan, Y.~Hu, and B.~Pasik-Duncan.
\newblock Stochastic calculus for fractional {B}rownian motion. {I}. {T}heory.
\newblock {\em SIAM J. Control Optim.}, 38(2):582--612 (electronic), 2000.

\bibitem{duncan2009control}
T.E. Duncan and B.~Pasik-Duncan.
\newblock Control of some linear stochastic systems with a fractional brownian
  motion.
\newblock In {\em Decision and Control, 2009 held jointly with the 2009 28th
  Chinese Control Conference. CDC/CCC 2009. Proceedings of the 48th IEEE
  Conference}, pages 8518--8522, 2009.

\bibitem{eh}
R.J. Elliott and J.~van~der Hoek.
\newblock A general fractional white noise theory and applications to finance.
\newblock {\em Math. Finance}, 13(2):301--330, 2003.

\bibitem{MR1244577}
T.~Hida, H.~Kuo, J.~Potthoff, and L.~Streit.
\newblock {\em White noise}, volume 253 of {\em Mathematics and its
  Applications}.
\newblock Kluwer Academic Publishers Group, Dordrecht, 1993.
\newblock An infinite-dimensional calculus.

\bibitem{Hille_Phillips}
Einar Hille and Ralph~S. Phillips.
\newblock {\em Functional analysis and semi-groups}.
\newblock American Mathematical Society, Providence, R. I., 1974.
\newblock Third printing of the revised edition of 1957, American Mathematical
  Society Colloquium Publications, Vol. XXXI.

\bibitem{MR1408433}
H.~Holden, B.~{\O}ksendal, J.~Ub{\o}e, and T.~Zhang.
\newblock {\em Stochastic partial differential equations}.
\newblock Probability and its Applications. Birkh\"auser Boston Inc., Boston,
  MA, 1996.

\bibitem{MR1976868}
Yaozhong Hu and Bernt {\O}ksendal.
\newblock Fractional white noise calculus and applications to finance.
\newblock {\em Infin. Dimens. Anal. Quantum Probab. Relat. Top.}, 6(1):1--32,
  2003.

\bibitem{hu2005stochastic}
Yaozhong Hu and Xun~Yu Zhou.
\newblock Stochastic control for linear systems driven by fractional noises.
\newblock {\em SIAM J. Control Optim.}, 43(6):2245--2277 (electronic), 2005.

\bibitem{MR99f:60082}
S.~Janson.
\newblock {\em Gaussian {H}ilbert spaces}, volume 129 of {\em Cambridge Tracts
  in Mathematics}.
\newblock Cambridge University Press, Cambridge, 1997.

\bibitem{KlepLeb2002}
M.~L. Kleptsyna and A.~Le~Breton.
\newblock Extension of the {K}alman-{B}ucy filter to elementary linear systems
  with fractional {B}rownian noises.
\newblock {\em Stat. Inference Stoch. Process.}, 5(3):249--271, 2002.

\bibitem{MR0012176}
M.G. Krein.
\newblock On the problem of continuation of helical arcs in {H}ilbert space.
\newblock {\em C. R. (Doklady) Acad. Sci. URSS (N.S.)}, 45:139--142, 1944.

\bibitem{MR1387829}
Hui-Hsiung Kuo.
\newblock {\em White noise distribution theory}.
\newblock Probability and Stochastics Series. CRC Press, Boca Raton, FL, 1996.

\bibitem{MR0004644}
{J. von} Neumann and I.~J. Schoenberg.
\newblock Fourier integrals and metric geometry.
\newblock {\em Trans. Amer. Math. Soc.}, 50:226--251, 1941.

\bibitem{MR2220074}
David Nualart and Murad~S. Taqqu.
\newblock Wick-{I}t\^o formula for {G}aussian processes.
\newblock {\em Stoch. Anal. Appl.}, 24(3):599--614, 2006.

\bibitem{MR2456333}
David Nualart and Murad~S. Taqqu.
\newblock Wick-{I}t\^o formula for regular processes and applications to the
  {B}lack and {S}choles formula.
\newblock {\em Stochastics}, 80(5):477--487, 2008.

\bibitem{sch_III}
L.~Schwartz.
\newblock {\em Analyse. {III}}, volume~44 of {\em Collection Enseignement des
  Sciences [Collection: The Teaching of Science]}.
\newblock Hermann, Paris, 1998.
\newblock Calcul int\'egral.

\end{thebibliography}

\def\cprime{$'$} \def\lfhook#1{\setbox0=\hbox{#1}{\ooalign{\hidewidth
  \lower1.5ex\hbox{'}\hidewidth\crcr\unhbox0}}} \def\cprime{$'$}
  \def\cprime{$'$} \def\cprime{$'$} \def\cprime{$'$} \def\cprime{$'$}

\end{document}